\documentclass[12pt,a4paper,reqno]{amsart}
\usepackage[T1]{fontenc}
\usepackage{latexsym,amsmath,amssymb,ifthen,mathabx}
\usepackage{fullpage,verbatim}
\usepackage{url}
\usepackage{graphicx,psfrag,subfigure}
\usepackage{color}

\newcommand{\R}{{\mathbb R}}

\newcommand{\N}{{\mathbb N}}

\newcommand{\Ht}[2]{\widetilde H^{#1}({#2})}
\newcommand{\Hs}[2]{H^{#1}({#2})}

\newcommand{\Hta}[2]{\wtd H_{\ast}^{#1}(#2)}

\def\bigset#1#2{\Big\{#1\,:\,#2\Big\}}
\newcommand{\inpro}[2]{\left\langle{#1},{#2}\right\rangle}

\newcommand{\norm}[2]{\left\|{#1}\right\|_{#2}}
\newcommand{\bnorm}[2]{\big\|{#1}\big\|_{#2}}

\newcommand{\snorm}[2]{\left|{#1}\right|_{#2}}
\newcommand{\bsnorm}[2]{\big|{#1}\big|_{#2}}
\newcommand{\Bsnorm}[2]{\Big|{#1}\Big|_{#2}}

\newcommand{\set}[2]{\left\{{#1} \, : \, {#2} \right\}}

\newcommand{\vecA}{\boldsymbol{A}}

\newcommand{\vecC}{\boldsymbol{C}}

\newcommand{\vecb}{\boldsymbol{b}}

\newcommand{\vecu}{\boldsymbol{u}}

\newcommand{\vecx}{\boldsymbol{x}}

\DeclareMathOperator*{\argmin}{argmin}

\DeclareMathOperator{\diam}{{diam }}
\DeclareMathOperator{\dist}{{dist }}

\DeclareMathOperator{\supp}{{supp \/}}

\newcommand{\hyp}{\mathfrak{W}}

\newcommand{\cA}{{\mathcal A}}

\newcommand{\cC}{{\mathcal C}}
\newcommand{\cD}{{\mathcal D}}

\newcommand{\cL}{{\mathcal L}}

\newcommand{\cO}{{\mathcal O}}

\newcommand{\cS}{{\mathcal S}}

\newcommand{\cV}{{\mathcal V}}
\newcommand{\cX}{{\mathcal X}}

\newcommand{\wtd}{\widetilde}
\newcommand{\wth}{\widehat}

\newcommand{\ol}{\overline}

\newcommand{\pa}{\partial}
\renewcommand{\emptyset}{\varnothing}

\newcommand{\dr}{\, {\rm d}r}
\newcommand{\ds}{\, {\rm d}s}
\newcommand{\dt}{\, {\rm d}t}
\newcommand{\dx}{\, {\rm d}x}
\newcommand{\dy}{\, {\rm d}y}

\newcommand{\rmd}{\, {\rm d}}

\usepackage[a4paper, margin=3.0cm]{geometry}

\numberwithin{equation}{section}
\newtheorem{theorem}{Theorem}[section]
\newtheorem{lemma}[theorem]{Lemma}
\newtheorem{corollary}[theorem]{Corollary}
\newtheorem{proposition}[theorem]{Proposition}
\newtheorem{remark}[theorem]{Remark}

\numberwithin{equation}{section}

\begin{document}
\title{
Remarks on Sobolev norms of fractional orders
}
\author{Thanh Tran}
\address{School of Mathematics and Statistics,
         The University of New South Wales,
         Sydney 2052, Australia}
\email{thanh.tran@unsw.edu.au}
\thanks{Supported by the Australian Research Council under grant number DP190101197}


\date{\today}

\begin{abstract}
When a function belonging to a fractional-order Sobolev space is supported in a proper subset of
the Lipschitz domain on which the Sobolev space is defined, how is its Sobolev norm as a function
on the smaller set compared to its norm on the whole domain? On what do the comparison
constants depend on? Do different norms behave differently? This article addresses these
issues. We prove some inequalities and disprove some misconceptions by counter-examples.
\end{abstract}
\maketitle

\section{Introduction}
Sobolev spaces of fractional orders are the key function space for the mathematical and
numerical analysis of boundary integral equation methods. It is well known that these spaces
behave differently with Sobolev spaces of integral orders. Consider, for example, a
domain~$\cO$ in~$\R^n$, $n\ge1$, which is partitioned by two subdomains~$\cO_1$ and~$\cO_2$ so 
that~$\ol\cO=\ol\cO_1\cup\ol\cO_2$ with $\cO_1\cap\cO_2=\emptyset$. If~$u\in H^m(\cO)$
for some non-negative integer~$m$, where~$H^m(\cO)=W_2^{m}(\cO)$ is the normally-defined
Sobolev space with the Lebesgue measure, then
\begin{equation}\label{equ:Hm O1 O2}
\norm{u}{H^m(\cO)}^2 = \norm{u|_{\cO_1}}{H^m(\cO_1)}^2 + \norm{u|_{\cO_2}}{H^m(\cO_2)}^2.
\end{equation}
It is also obvious that if~$u\in H^m(\cO)$ with~$\supp(u)\subset\ol\cO_1$, then
\begin{equation}\label{equ:Hm O1}
\norm{u}{H^m(\cO_1)} = \norm{u}{H^m(\cO)}.
\end{equation}
Are the above properties true for fractional-order Sobolev spaces? 

Pushing back until the next section the precise definitions of the fractional-order Sobolev
spaces~$\Ht{s}{\cO}$ and~$\wtd W_2^s(\cO)$, $s>0$, we mention here that these two
spaces coincide when~$s-1/2\in\N$, the space of non-negative integers; see
Subsection~\ref{subsec:equ nor}. Their corresponding
norms, denoted by~$\norm{\cdot}{\Ht{s}{\cO}}$ and~$\norm{\cdot}{\sim,s,\cO}$, respectively, are
equivalent, i.e., there exist positive constants~$C_1$ and~$C_2$ satisfying
\[
C_1 \norm{u}{\Ht{s}{\cO}} \le \norm{u}{\sim,s,\cO} \le C_2 \norm{u}{\Ht{s}{\cO}} 
\quad\forall u\in\Ht{s}{\cO}.
\]
We first make a remark (Proposition~\ref{pro:A 2 nor 3 nor}) on how the constants~$C_1$
and~$C_2$ depend on the size of the domain~$\cO$.

An issue with the Sobolev norm~$\norm{\cdot}{\Ht{s}{\cO}}$ is that instead of~\eqref{equ:Hm O1
O2} it is known that
\begin{equation}\label{equ:Hs O1 O2}
\norm{u}{\Ht{s}{\cO}}^2 \le 
\norm{u|_{\cO_1}}{\Ht{s}{\cO_1}}^2 + \norm{u|_{\cO_2}}{\Ht{s}{\cO_2}}^2,
\end{equation}
provided that all norms are well defined.
This result is proved in~\cite{AinMcLTra,Pet}. A counter-example in~\cite{AinGuo00} shows that
the opposite inequality is not true. This means there is no universal constant~$C$ independent
of~$u$, $\cO_1$, $\cO_2$, and~$\cO$ such that
\begin{equation}\label{equ:wro}
\norm{u|_{\cO_1}}{\Ht{s}{\cO_1}}^2 + \norm{u|_{\cO_2}}{\Ht{s}{\cO_2}}^2
\le C \norm{u}{\Ht{s}{\cO}}^2.
\end{equation}
Inequality~\eqref{equ:Hs O1 O2} and the non-existence of~\eqref{equ:wro} imply that
if~$u\in\Ht{s}{\cO}$ is supported in~$\ol\cO_1$ then, instead of~\eqref{equ:Hm
O1}, we have in general
\[
\norm{u}{\Ht{s}{\cO}} < \norm{u}{\Ht{s}{\cO_1}}.
\]
We establish in Theorem~\ref{the:A Sob loc nor} that when~$s=1/2$ (a commonly-seen case)
\[
\norm{u}{\sim,s,\cO} \le C \norm{u}{\sim,s,\cO_1}
\]
but there is no constant~$C'$ independent of~$\cO$ and~$\cO_1$ 
satisfying
\[
\norm{u}{\sim,s,\cO_1} \le C' \norm{u}{\sim,s,\cO}
\]
for all~$u\in\wtd W_2^s(\cO)$ supported in~$\cO_1$.

A third issue arises from the analysis of domain decomposition methods for boundary integral
equations of the first kind. Consider for example the hypersingular integral equation; see
Section~\ref{sec:app} for detail. It is known that the bilinear form~$a(\cdot,\cdot)$
arising from this operator defines a norm equivalent to the $\Ht{1/2}{\cO}$-norm. One of the
requirements in the analysis is a proof of some inequality of the form
\begin{equation}\label{equ:a O1 O2}
a(u,u) \le C \big( a(u|_{\cO_1},u|_{\cO_1}) + a(u|_{\cO_2},u|_{\cO_2}) \big).
\end{equation}
Due to a misconception that the norm~$\sqrt{a(u|_{\cO_j},u|_{\cO_j})}$ is equivalent
to~$\norm{u|_{\cO_j}}{\Ht{1/2}{\cO_j}}$, $j=1,2$, inequality~\eqref{equ:Hs O1 O2} has
been used ubiquitously in the literature to obtain the above estimate. Proposition~\ref{pro:A 2
nor 3 nor} and Theorem~\ref{the:A Sob loc nor} imply that this equivalence is at the
cost of the equivalence constants depending on the size of the subdomain~$\cO_j$. This
may adversely affect the final result; see Section~\ref{sec:app} for detail. We prove in
Theorem~\ref{the:A tvp 2} that~\eqref{equ:Hs O1 O2} can be improved to ensure the
following estimate
\[
\norm{u}{\Ht{s}{\cO}}^2 \le 
\norm{u|_{\cO_1}}{\Ht{s}{\cO}}^2 + \norm{u|_{\cO_2}}{\Ht{s}{\cO}}^2,
\]
assuming that~$u|_{\cO_j}$ is extended by zero to the exterior of~$\cO_j$, $j=1,2$.
This inequality is the right tool to prove~\eqref{equ:a O1 O2}.

The remainder of the article is organised as follows. In Section~\ref{sec:Sob nor} we present
the precise definitions of the Sobolev spaces and norms in consideration. The main results are
stated in Section~\ref{sec:mai res}, the proofs of which are performed in Section~\ref{sec:pro}.
Section~\ref{sec:app} discusses applications of these results.

In the sequel, if~$a \le c b$ where the constant~$c$ is independent of the parameters in
concern, for example, the functions and the size of the domain, we will write~$a \lesssim b$.
We will also write~$a \simeq b$ if~$a \lesssim b$ and~$b \lesssim a$.
\section{Sobolev norms}\label{sec:Sob nor}
In this section we recall the definitions of Sobolev spaces of fractional orders that we
are interested in.
Let~$\cO$ be a generic bounded and connected domain in~$\R^n$, $n\ge1$, with Lipschitz boundary,
and let~$r$ be a positive integer. 
The spaces~$L^2(\cO)$ is the space of Lebesgue squared-integrable functions defined on~$\cO$,
with norm denoted by~$\norm{\cdot}{0,\cO}$. The space $W_2^r(\cO)=H^r(\cO)$ is defined by
\begin{align*}
W_2^r(\cO) &= H^r(\cO)
:= \set{u\in L^2(\cO)}{\norm{u}{r,\cO}<\infty}
\end{align*}
where
\begin{align*}
\norm{u}{r,\cO}
&:=
\Big( \sum_{|\alpha|=0}^r \norm{\pa^\alpha u}{0,\cO}^2 \Big)^{1/2}.
\end{align*}
Here, $\alpha=(\alpha_1,\cdots,\alpha_n)\in\N^n$ is a multi-index,
$|\alpha|=\alpha_1+\cdots+\alpha_n$,
and~$\pa^\alpha u=\pa^{|\alpha|}u/\pa_{x_1}^{\alpha_1}\cdots\pa_{x_n}^{\alpha_n}$ is the
$|\alpha|$-order partial derivative of~$u$. 

The space~$\mathring{W}_2^r(\cO)=H_0^r(\cO)$ is defined as the closure
of~$\cD(\cO)$ in~$H^r(\cO)$, where~$\cD(\cO)$ is the space of all~$\cC^\infty$ functions
with compact support in~$\cO$. Thanks to Poincar\'e's inequality, the seminorm
\[
\snorm{u}{r,\cO}
:=
\Big( \sum_{|\alpha|=r} \norm{\pa^\alpha u}{0,\cO}^2 \Big)^{1/2}
\]
can be used as a norm on~$H_0^r(\cO)$.

In the sequel, we define the fractional-order Sobolev spaces of order~$s>0$; see e.g.
\cite{Ada75,Gri85}.

\subsection{Real interpolation spaces~$\Ht{s}{\cO}$ for $s>0$}\label{subsec:int spa}
The general method for constructing real interpolation spaces can be found in~\cite{BerLof}.
For any~$u\in L^2(\cO)$ and any~$t>0$, the functional~$K(t,u)$ is defined by
\begin{equation}\label{equ:K fun}
K(t,u) := \inf_{(u_0,u_1)\in\cX(u)}
\left(\norm{u_0}{0,\cO}^2 + t^2\snorm{u_1}{r,\cO}^2\right)^{1/2}
\end{equation}
where
\begin{equation}\label{equ:Xu}
\cX(u) = \set{(u_0,u_1)\in L^2(\cO) \times H_0^r(\cO) }{u_0+u_1=u}.
\end{equation}
For~$s\in(0,r)$, with~$\theta=s/r$ we define the interpolation
space~$[L^2(\cO),H_0^r(\cO)]_{\theta}$ by
\[
[L^2(\cO),H_0^r(\cO)]_{\theta} :=
\set{u\in L^2(\cO)}{\norm{u}{[L^2(\cO),H_0^r(\cO)]_{\theta}} < \infty}
\]
where
\[
\norm{u}{[L^2(\cO),H_0^r(\cO)]_{\theta}} :=
\left(
\int_0^\infty |t^{-\theta}K(t,u)|^2 \frac{\dt}{t}
\right)^{1/2}.
\]
We follow~\cite{Gri85} to denote this space by~$\Ht{s}{\cO}$ and equipped it with the norm
\[
\norm{u}{\Ht{s}{\cO}} := \norm{u}{[L^2(\cO),H_0^r(\cO)]_{\theta}}.
\]

If~$s-1/2\notin\N$ (the set of non-negative integers) then~$\wtd H^s(\cO)=H_0^s(\cO)$, the closure
of~$\cD(\cO)$ in~$H^s(\cO)$. Here, the space~$\Hs{s}{\cO}:=[L^2(\cO),H^r(\cO)]_{\theta}$
is defined by interpolation as is~$\Ht{s}{\cO}$, with the $K$-functional defined by using the
norm~$\norm{u_1}{r,\cO}$ instead of the seminorm~$\snorm{u_1}{r,\cO}$.

If~$s-1/2\in\N$, then $\Ht{s}{\cO}$ is a proper subset
of~$H_0^s(\cO)$ and is denoted by~$H_{00}^s(\cO)$ in~\cite{LioMag72}. We follow~\cite{Gri85} to
use the same notation~$\wtd H^s(\cO)$ in both cases.

A special case is when~$u\in\Ht{s}{\cO}$ is supported in~$\ol\cO_1$, where~$\cO_1$ is a bounded
Lipschitz domain which is a proper subset of~$\cO$. The function~$u$ also belongs
to~$\Ht{s}{\cO_1}$. We want to compare the norm~$\norm{u}{\Ht{s}{\cO_1}}$
with~$\norm{u}{\Ht{s}{\cO}}$.

First we clarify 
how the norm~$\norm{u}{\Ht{s}{\cO_1}}$ is defined.
For~$r\in\N\setminus\{0\}$, we define two spaces
\begin{align*}
A_0 &:= \set{u\in L^2(\cO)}{\supp(u)\subseteq\ol\cO_1},
\\
A_1 &:= \set{u\in H_0^r(\cO)}{\supp(u)\subseteq\ol\cO_1},
\end{align*}
which form a compatible couple~$\cA=(A_0,A_1)$. By zero extension, we can
identify~$L^2(\cO_1)$ and~$H_0^r(\cO_1)$ with~$A_0$ and~$A_1$, respectively. For
any~$u\in A_0$ we define, similarly to~\eqref{equ:Xu},
\begin{equation}\label{equ:Xu1}
\cX_1(u) := \set{(u_0,u_1)\in A_0 \times A_1}{u_0+u_1=u}.
\end{equation}

We also define two functionals~$J : \R^+ \times L^2(\cO)\times H_0^r(\cO)\to\R$
and~$J_1 : \R^+ \times A_0 \times A_1 \to \R$ by
\begin{equation}\label{equ:A JJ1}
\begin{alignedat}{2}
J(t,u_0,u_1) &:= \norm{u_0}{0,\cO}^2 + t^2\snorm{u_1}{r,\cO}^2, &\quad
& t>0, \ u_0 \in L^2(\cO), \ u_1 \in H_0^r(\cO),
\\
J_1(t,u_0,u_1) &:= \norm{u_0}{0,\cO_1}^2 + t^2\snorm{u_1}{r,\cO_1}^2, &\quad
& t>0, \ u_0 \in A_0, \ u_1 \in A_1.
\end{alignedat}
\end{equation}
The $K$-functional defined in~\eqref{equ:K fun} can be written as
\begin{equation}\label{equ:Ktu}
K(t,u) = \inf\set{J^{1/2}(t,u_0,u_1)}{(u_0,u_1)\in\cX(u)}.
\end{equation}
Correspondingly, we define
\begin{equation}\label{equ:K1tu}
K_1(t,u) = \inf\set{J^{1/2}_1(t,u_0,u_1)}{(u_0,u_1)\in\cX_1(u)}
\end{equation}
and the corresponding norm
\begin{equation}\label{equ:A0 A1 nor}
\norm{u}{[A_0,A_1]_{\theta}} :=
\left(
\int_0^\infty |t^{-\theta}K_1(t,u)|^2 \frac{\dt}{t}
\right)^{1/2}, \quad \theta\in(0,1).
\end{equation}

We note that~$\cX_1(u)$ is a proper subset of~$\cX(u)$ because in the definition
of~$\cX(u)$ for~$u\in A_0$, the two functions~$u_0$ and~$u_1$ do not have to be zero
in~$\cO_2:=\cO\setminus\ol\cO_1$, but $u_0=-u_1$ in~$\ol\cO_2$. Consequently, for any~$t>0$,
\[
K(t,u) \le K_1(t,u).
\]
We will prove later that in general this is indeed a strict inequality.

Using the norm defined by~\eqref{equ:A0 A1 nor} we can define the interpolation space
\[
[A_{0},A_{1}]_{\theta} :=
\set{u\in A_{0}}{\norm{u}{[A_{0},A_{1}]_{\theta}}<\infty}.
\]
This space is the usual space~$\Ht{s}{\cO_1}$ which is equipped with the
norm~$\norm{u}{\Ht{s}{\cO_1}}=\norm{u}{[A_{0},A_{1}]_{\theta}}$, where~$s=\theta r$.
Denote
\[
\wtd H_{\ast}^{s}(\cO_1) :=
\set{u\in A_0}{\norm{u}{\Ht{s}{\cO}} < \infty}
\quad\text{and}\quad
\norm{u}{\wtd H_{\ast}^s(\cO_1)} := \norm{u}{\Ht{s}{\cO}}.
\]
Clearly~$\Hta{s}{\cO_1}$ is a proper subset of~$\Ht{s}{\cO}$.
We will prove in Subsection~\ref{subsec:pro inc} that the following subset inclusion
is proper
\begin{equation}\label{equ:A Hs Hss}
\wtd H^{s}(\cO_1)\subsetneq\wtd H_{\ast}^{s}(\cO_1). 
\end{equation}

\subsection{Sobolev--Slobodetski spaces~$W_2^s(\cO)$ and~$\wtd W_2^s(\cO)$ for $s>0$}
Let~$s=m+\sigma$ with~$m\in\N$ and~$\sigma\in(0,1)$. For every function~$u$
defined in~$\cO$, we define
\begin{align*}
\norm{u}{m,\cO} &:= 
\Big( \displaystyle\sum_{|\alpha|=0}^m \norm{\pa^\alpha u}{0,\cO}^2 \Big)^{1/2},
\\[1ex]
\snorm{u}{\sigma,\cO} &:= 
\Big(
\iint_{\cO\times \cO} \frac{|u(x)-u(y)|^2}{|x-y|^{n+2\sigma}}\dx\dy 
\Big)^{1/2}.
\\[1ex]
\norm{u}{s,\cO} &:=
\Big( \norm{u}{m,\cO}^2 + \sum_{|\alpha|=m} \snorm{\pa^\alpha u}{\sigma,\cO}^2 \Big)^{1/2}.
\end{align*}
The space~$W_2^s(\cO)$ is the space of all functions~$u$ defined in~$\cO$ such that
\(
\norm{u}{s,\cO} < \infty.
\)
The space~$\mathring{W}_2^s(\cO)$ is defined to be the closure of~$\cD(\cO)$
in~$W_2^s(\cO)$. Using this space, we define
\[
\wtd W_2^s(\cO) :=
\set{u\in\mathring{W}_2^s(\cO)}{\pa^\alpha u/\rho^\sigma\in L^2(\cO), \ |\alpha|=m}
\]
where~$\rho(x):=\dist(x,\partial \cO)$ is the distance from~$x$ to the
boundary~$\partial \cO$ of~$\cO$. This space is equipped with the norm
\[
\norm{u}{\sim,s,\cO} :=
\Big(
\norm{u}{s,\cO}^2 + 
\sum_{|\alpha|=m}\int_\cO \frac{|\pa^\alpha u(x)|^2}{\rho^{2\sigma}(x)} \dx
\Big)^{1/2}.
\]
We note that when~$s=\sigma\in(0,1)$, i.e., $m=0$, the norm can also be defined by
\begin{equation}\label{equ:m0}
\norm{u}{\sim,s,\cO} := 
\Big(
\snorm{u}{s,\cO}^2 + 
\int_\cO \frac{|u(x)|^2}{\rho^{2\sigma}(x)} \dx
\Big)^{1/2}.
\end{equation}
The advantage of this norm is that the two terms defining the norm scale similarly when the
domain~$\cO$ is rescaled; see Proposition~\ref{pro:A 2 nor 3 nor}.

\subsection{Equivalence of norms}\label{subsec:equ nor}
If~$s=m+1/2$ for~$m\in\N$, then~$\Ht{s}{\cO}=\wtd W_2^s(\cO)$;
see~\cite[Theorem~11.7, page~66]{LioMag72}. Moreover,
\begin{equation}\label{equ:Lio Mag}
\norm{u}{\Ht{s}{\cO}} \simeq \norm{u}{\sim,s,\cO}.
\end{equation}

It is sometimes important to know how the constants depend on the size of~$\cO$.
In the sequel, we assume that all domains are ``shape regular'', i.e., we avoid long and thin
shapes. We first explain the motivation of this assumption, via an example, and give a more
precise description of ``shape-regular'' domains.

Let~$\Omega$ be a subset of~$\R^3$ which defines the geometry for equation~\eqref{equ:hyp} in
Section~\ref{sec:app}. A simple example is~$\Omega=(0,1)\times(0,1)\times\{0\}$.
The solution~$\varphi$ to equation~\eqref{equ:hyp} is sought numerically by
finding~$\varphi_M$ which solves~\eqref{equ:Gal equ}. The numerical scheme is designed
such that~$\varphi_M$ converges to~$\varphi$ in~$\Ht{1/2}{\Omega}$ as~$M\to\infty$.

In the boundary element approximation, we compute~$\varphi_M$ by
first partitioning~$\Omega$ into subdomains~$\Omega_1$, \ldots, $\Omega_M$ (called
boundary elements) which can be
rectangles, quadrilaterals, or triangles, and then compute~$\varphi_M$ as a continuous
piecewise polynomial function on this partition. The partitioning is repeated during the
solution process to increase accuracy of the approximation. Note 
that~$\max_{1\le j\le M}\diam(\Omega_j)\to0$ as~$M\to\infty$. Very often in practice,
the partitioning is carried out
in such a way that there exists a constant~$C$ independent of~$M$ satisfying
\begin{equation}\label{equ:sha reg}
\max_{1\le j\le M}\frac{\diam(\Omega_j)}{\varrho(\Omega_j)} \le C,
\end{equation}
where~$\varrho(\Omega_j)$ is the radius of the largest inscribed ball in~$\Omega_j$. Such
domains~$\Omega_j$ are called shape-regular boundary elements. A simple partition
for~$\Omega$ mentioned above which satisfies~\eqref{equ:sha reg} is a uniform partition. 

With these applications in mind, in the remainder of this paper, we are only concerned
with the diameters of the subdomains. A study involving the constant~$C$
in~\eqref{equ:sha reg}, or involving both~$\diam(\Omega_j)$ and~$\varrho(\Omega_j)$
when~\eqref{equ:sha reg} is not satisfied, is possible
but it is not in the interest of this paper. Techniques for this consideration
can be found in~\cite{DupSco80,Heu14}. It should be noted
that~$\Omega$ is the given geometry and therefore~$\diam(\Omega)$ is fixed in the whole
solution process. We are not concerned with this quantity.

In the next proposition, we denote by~$\cO$ any subdomain~$\Omega_j$ satisfying~\eqref{equ:sha
reg}. We study the scaling property of the two norms~$\norm{\cdot}{\Ht{s}{\cO}}$
and~$\norm{\cdot}{\sim,s,\cO}$, and show how this property depends on~$\diam(\cO)$,
which can be assumed to be smaller than~$1$ because~$\diam(\Omega_j)\to0$ as described
above. 

\begin{proposition}\label{pro:A 2 nor 3 nor}
Assume that $\cO$ is a domain in~$\R^n$ with Lipschitz boundary, satisfying
$\tau:=\diam(\cO)<1$. The following statements hold true.
\begin{enumerate}
\renewcommand{\labelenumi}{\theenumi}
\renewcommand{\theenumi}{{\rm (\roman{enumi})}}
\item
If~$s=m+1/2$ with~$m=1,2,\ldots$, then
\[
\tau^{2s}
\norm{u}{\Ht{s}{\cO}}^2
\lesssim
\norm{u}{\sim,s,\cO}^2
\lesssim
\norm{u}{\Ht{s}{\cO}}^2
\quad\forall u \in \Ht{s}{\cO}.
\]
\item
If $s=1/2$ and if we define~$\norm{u}{\sim,1/2,\cO}$ by~\eqref{equ:m0}, i.e.,
\begin{equation}\label{equ:H12}
\norm{u}{\sim,1/2,\cO}
=
\Big(
\iint_{\cO\times \cO} \frac{|u(x)-u(y)|^2}{|x-y|^{n+1}}\dx\dy 
+
\int_\cO \frac{|u(x)|^2}{\rho(x)} \dx
\Big)^{1/2},
\end{equation}
then
\begin{equation}\label{equ:H12 Sob}
\norm{u}{\Ht{1/2}{\cO}} \simeq
\norm{u}{\sim,1/2,\cO}
\quad\forall u \in \Ht{1/2}{\cO}.
\end{equation}
\end{enumerate}
The constants are independent of $u$ and $\tau$.
\end{proposition}
\begin{proof}
First we consider~$s=m+\sigma$ for~$m\in\N$ and~$\sigma\in(0,1)$, and show how each norm
scales when the domain~$\cO$ is rescaled.
Let~$\wth\cO$ be a reference set of diameter~$1$ satisfying
\[
\wth x\in\wth\cO \iff \wth x = x/\tau, \quad x\in\cO,
\]
and let~$\wth u:\wth\cO\to\R$ be defined by~$\wth u(\wth x) = u(x)$ for
all~$\wth x\in\wth\cO$ and~$x\in\cO$. Simple calculations reveal
\[
\norm{u}{0,\cO}^2
=
\tau^n \norm{\wth u}{0,\wth\cO}^2
\quad\text{and}\quad
\pa_{\wth x}^\alpha \wth u(\wth x)
=
\tau^{|\alpha|}
\widehat{\pa_{x}^\alpha u}(\wth x)
\]
Hence, for~$m\in\N$,
\begin{align*}
\norm{u}{m,\cO}^2
&=
\sum_{|\alpha|=0}^m \norm{\pa_x^{\alpha}u}{0,\cO}^2
=
\tau^n
\sum_{|\alpha|=0}^m \bnorm{\wth{\pa_x^{\alpha}u}}{0,\wth\cO}^2
=
\tau^{n}
\sum_{|\alpha|=0}^m \tau^{-2|\alpha|}
\norm{\pa_{\wth x}^{\alpha}\wth{u}}{0,\wth\cO}^2.
\end{align*}
On the other hand, for~$\sigma\in(0,1)$ and~$|\alpha|=m$
\begin{align*}
\iint_{\cO\times\cO}
\frac{\snorm{\pa_x^\alpha u(x)-\pa_x^\alpha u(y)}{}^2}%
{|x-y|^{n+2\sigma}}
\dx \dy
&=
\iint_{\wth\cO\times\wth\cO}
\frac{\tau^{-2m} \bsnorm{\wth{\pa_{x}^\alpha u}(\wth x)
-\wth{\pa_{x}^\alpha u}(\wth y)}{}^2}%
{\tau^{n+2\sigma}|\wth x-\wth y|^{n+2\sigma}}
\, \tau^{2n} \rmd\wth x \, \rmd\wth y
\\
&=
\tau^{n-2s}
\iint_{\wth\cO\times\wth\cO}
\frac{\snorm{\pa_{\wth x}^\alpha\wth u(\wth x)
-\pa_{\wth x}^\alpha\wth u(\wth y)}{}^2}%
{|\wth x-\wth y|^{n+2\sigma}}
\, \rmd\wth x \, \rmd\wth y
\end{align*}
and, since~$\rho(x)=\dist(x,\partial\cO)=\tau\dist(\wth x,\partial\wth\cO)=\tau\wth\rho(\wth x)$,
\begin{align*}
\int_\cO \frac{|\pa_{x}^\alpha u(x)|^2}{\rho^{2\sigma}(x)} \dx
&=
\int_{\wth\cO} \frac{\tau^{-2m}
|\pa_{\wth x}^\alpha\wth u(\wth x)|^2}{\tau^{2\sigma}\wth\rho^{2\sigma}(\wth x)} 
\tau^{n} \rmd\wth x
=
\tau^{n-2s}
\int_{\wth\cO} \frac{|\pa_{\wth x}^\alpha\wth u(\wth x)|^2}{\wth\rho^{2\sigma}(\wth x)} 
\rmd\wth x.
\end{align*}
Consequently,
\begin{align*}
\norm{u}{\sim,s,\cO}^2
&=
\norm{u}{m,\cO}^2
+
\sum_{|\alpha|=m}
\iint_{\cO\times\cO}
\frac{\snorm{\pa_x^\alpha u(x)-\pa_x^\alpha u(y)}{}^2}%
{|x-y|^{n+2\sigma}}
\dx \dy
+
\sum_{|\alpha|=m}
\int_\cO \frac{|\pa_{x}^\alpha u(x)|^2}{\rho^{2\sigma}(x)} \dx
\\
&=
\tau^{n}
\sum_{|\alpha|=0}^m \tau^{-2|\alpha|}
\norm{\pa_{\wth x}^{\alpha}\wth{u}}{0,\wth\cO}^2
+
\tau^{n-2s}
\sum_{|\alpha|=m}
\iint_{\wth\cO\times\wth\cO}
\frac{\snorm{\pa_{\wth x}^\alpha\wth u(\wth x)
-\pa_{\wth x}^\alpha\wth u(\wth y)}{}^2}%
{|\wth x-\wth y|^{n+2\sigma}}
\, \rmd\wth x \, \rmd\wth y
\\
&\quad
+
\tau^{n-2s}
\sum_{|\alpha|=m}
\int_{\wth\cO} \frac{|\pa_{\wth x}^\alpha\wth u(\wth x)|^2}{\wth\rho^{2\sigma}(\wth x)} 
\rmd\wth x.
\end{align*}
Therefore,
\begin{equation}\label{equ:sca W}
\tau^n \norm{\wth u}{\sim,s,\wth \cO}^2
\le
\norm{u}{\sim,s,\cO}^2
\le
\tau^{n-2s} \norm{\wth u}{\sim,s,\wth \cO}^2.
\end{equation}
When~$m=0$ if we define the $\norm{\cdot}{\sim,s,\cO}$-norm by~\eqref{equ:H12} then
\begin{equation}\label{equ:s12}
\norm{u}{\sim,s,\cO}^2
=
\tau^{n-1} \norm{\wth u}{\sim,s,\wth \cO}^2.
\end{equation}

For the interpolation norm, we have
\[
\norm{u}{L^2(\cO)}^2 = \tau^n \norm{\wth u}{L^2(\wth\cO)}^2
\]
and
\[
\norm{u}{H_0^r(\cO)}^2 
= \sum_{|\alpha|=r} \norm{\pa^\alpha u}{L^2(\cO)}^2 
= \tau^{n-2r} \sum_{|\alpha|=r} \norm{\pa_{\wth x}^\alpha \wth u}{L^2(\wth\cO)}^2 
= \tau^{n-2r} \norm{\wth u}{H_0^r(\wth \cO)}^2.
\]
By interpolation
\begin{equation}\label{equ:sca H}
\norm{u}{\Ht{s}{\cO}}^2 = \tau^{n-2s} \norm{\wth u}{\Ht{s}{\wth\cO}}^2,
\quad 0\le s\le r.
\end{equation}

Now consider the case when~$s=m+1/2$. Using~\eqref{equ:Lio Mag}, \eqref{equ:sca W},
and~\eqref{equ:sca H}, we deduce
\begin{align*}
\norm{u}{\sim,s,\cO}^2
\le
\tau^{n-2s} \norm{\wth u}{\sim,s,\wth \cO}^2
\simeq
\tau^{n-2s} \norm{\wth u}{\Ht{s}{\wth\cO}}^2
=
\norm{u}{\Ht{s}{\cO}}^2
\end{align*}
and
\begin{align*}
\norm{u}{\Ht{s}{\cO}}^2
=
\tau^{n-2s} \norm{\wth u}{\Ht{s}{\wth\cO}}^2
\simeq
\tau^{n-2s} \norm{\wth u}{\sim,s,\wth \cO}^2
\le
\tau^{-2s} \norm{u}{\sim,s,\cO}^2,
\end{align*}
yielding the first part of the lemma. The constants in the above equivalences~$\simeq$
are the constants in~\eqref{equ:Lio Mag}, which depend on the size of~$\wth\cO$.
Recall that~$\diam(\wth\cO)=1$.

In the case when~$s=1/2$ with norm defined
by~\eqref{equ:H12}, it follows from~\eqref{equ:s12} and~\eqref{equ:sca H} that
\[
\norm{u}{\Ht{1/2}{\cO}}^2
=
\tau^{n-1} \norm{\wth u}{\Ht{1/2}{\wth\cO}}^2
\simeq
\tau^{n-1} \norm{\wth u}{\sim,1/2,\wth \cO}^2
=
\norm{u}{\sim,1/2,\cO}^2,
\]
completing the proof of the lemma.
\end{proof}

The following theorem concerning properties of the $\Ht{s}{\Omega}$ norms is proved in
\cite[Lemma 3.2]{Pet} and in~\cite[Theorem~4.1]{AinMcLTra}.
\begin{theorem} \label {the:A tvp}
Let $\{\Omega_1, \ldots,\Omega_N\}$ be a partition of a bounded Lipschitz
domain~$\Omega$ into non-overlapping Lipschitz domains. For~$0\le s\le r$, the
following inequalities hold (assuming that all the norms are well defined)
\begin{equation}\label {equ:A tvp2}
\|u\|_{\wtd H^s(\Omega)}^2 
\le 
\sum_{j=1}^N \|u|_{\Omega_j}\|_{\wtd H^s(\Omega_j)}^2.
\end{equation}
\end{theorem}

A consequence of the above theorem is that, under the assumption of Theorem~\ref{the:A tvp},
if~$\supp(u)\subset\ol\Omega_1$  then
\begin{equation}\label{equ:Ome Omep}
\|u\|_{\wtd H^s(\Omega)}^2 
\le 
\|u|_{\Omega_1}\|_{\wtd H^s(\Omega_1)}^2,
\end{equation}
provided that both norms are well defined. 
\section{The main results}\label{sec:mai res}

We now state our main results, the proofs of which will be carried out in
Section~\ref{sec:pro}. The first theorem confirms that if~$u\in\wtd
H^{1/2}(\cO)$ is such that~$\supp(u)\subset\ol\cO'$ where~$\cO'$ is a proper
subset of~$\cO$, which is itself a shape-regular Lipschitz domain,
then the two norms~$\norm{u}{\sim,1/2,\cO}$
and~$\norm{u}{\sim,1/2,\cO'}$ are not equivalent.

\begin{theorem}\label{the:A Sob loc nor}
Let~$\cO$ be a bounded Lipschitz domain in~$\R^n$, $n\ge1$.
Assume that~$u\in\Ht{1/2}{\cO}$ satisfies~$\supp(u)\subset\ol\cO'\subsetneq\cO$.
\begin{enumerate}
\renewcommand{\labelenumi}{\theenumi}
\renewcommand{\theenumi}{{\rm (\roman{enumi})}}
\item\label{ite:A nor Ome Omep}
The following relations between norms of~$u$ hold
\[
\norm{u}{\sim,1/2,\cO}
\le
C \norm{u}{\sim,1/2,\cO'}
\]
where~$C=(2\omega_{n}+1)^{1/2}$ with~$\omega_n$ being the surface area of the unit
ball in~$\R^n$.
\item\label{ite:A rev ine}
The opposite inequality is not true, i.e., there is no constant~$c$ independent of~$u$,
$\diam(\cO')$, and~$\diam(\cO)$ such that
\(
\norm{u}{\sim,1/2,\cO'}
\le c
\norm{u}{\sim,1/2,\cO}
\).
\end{enumerate}
\end{theorem}

Our next main result improves Theorem~\ref{the:A tvp}, namely we prove that the norm on the
right-hand side of~\eqref{equ:A tvp2} can be replaced
by~$\norm{u|_{\Omega_j}}{\Ht{s}{\Omega}}^2$; cf. \eqref{equ:Ome Omep}.

\begin{theorem} \label{the:A tvp 2}
Let $\{\Omega_1, \ldots,\Omega_N\}$ be a partition of a bounded Lipschitz
domain~$\Omega$ into non-overlapping Lipschitz domains. For~$0\le s\le r$, 
let~$u\in\wtd H^{s}(\Omega)$ be such that~$u_j\in\Hta{s}{\Omega_j}$,
where~$u_j$ is the zero extension of~$u|_{\Omega_j}$ onto~$\Omega\setminus\ol\Omega_j$,
$j=1,\ldots,N$.
Then the following inequality holds
\begin{equation}\label {equ:A tvp3}
\norm{u}{\Ht{s}{\Omega}}^2 
\le 
\sum_{j=1}^N \bnorm{u_{j}}{\Ht{s}{\Omega}}^2.
\end{equation}
\end{theorem}

A direct consequence of Theorem~\ref{the:A tvp 2} is the following corollary which
generalises Theorem~\ref{the:A tvp} and has applications discussed in Section~\ref{sec:app}.

\begin{corollary}\label{cor:tvp}
Under the assumption of Theorem~\ref{the:A tvp} we have
\begin{equation}\label{equ:A tvp 3}
\norm{u}{\Ht{1/2}{\Omega}}^2
\le
\sum_{j=1}^N \norm{u_{j}}{\Ht{1/2}{\Omega}}^2
\end{equation}
where~$u_j$ is the zero extension of~$u|_{\Omega_j}$ onto~$\Omega\setminus\ol\Omega_j$.
\end{corollary}
\begin{proof}
The result is a direct consequence of Theorem~\ref{the:A tvp 2}, noting
that~$\Ht{1/2}{\Omega_j}\subset\Hta{1/2}{\Omega_j}$, $j=1,\ldots,N$.
\end{proof}

\section{Proofs of the main results}\label{sec:pro}

\subsection{Proof of Theorem~\ref{the:A Sob loc nor}}
\begin{proof}
We first prove part~\ref{ite:A nor Ome Omep}.
Recall the definition of~$\norm{v}{\sim,1/2,\cO}$:
\[
\norm{u}{\sim,1/2,\cO}^2
=
\snorm{u}{1/2,\cO}^2
+
\int_{\cO} \frac{|u(x)|^2}{\dist(x,\partial\cO)} \dx.
\]
Clearly,
\begin{equation}\label{equ:A dis Ome Ome'}
\int_{\cO} \frac{|u(x)|^2}{\dist(x,\partial\cO)} \dx
=
\int_{\cO'} \frac{|u(x)|^2}{\dist(x,\partial\cO)} \dx
\le
\int_{\cO'} \frac{|u(x)|^2}{\dist(x,\partial\cO')} \dx.
\end{equation}
On the other hand, since~$\supp(u)\subset\ol\cO'$
\begin{equation}\label{equ:u12 Omg Omp}
\snorm{u}{1/2,\cO}^2
=
\iint_{\cO'\times\cO'}
\frac{|u(x)-u(y)|^2}{|x-y|^{n+1}} \dx \dy
+
2 \int_{\cO'}\Big(
\int_{\cO\setminus\cO'} \frac{\dy}{|x-y|^{n+1}}
\Big) |u(x)|^2 \dx.
\end{equation}
It will be proved in Lemma~\ref{lem:A Ome Omep} below that
\begin{equation}\label{equ:A Ome Omep dis}
\int_{\cO\setminus\cO'} \frac{\dy}{|x-y|^{n+1}}
\le
\frac{\omega_{n}}{\dist(x,\pa\cO')}, \quad x\in\cO',
\end{equation}
where~$\omega_n$ is the surface area of the unit ball in~$\R^n$. Hence
\[
\snorm{u}{1/2,\cO}^2
\le
\snorm{u}{1/2,\cO'}^2
+
2\omega_{n} \int_{\cO'} \frac{|u(x)|^2}{\dist(x,\pa\cO')} \dx
\le
2\omega_{n} \norm{u}{\sim,1/2,\cO'}^2
\]
so that, with the help of~\eqref{equ:A dis Ome Ome'},
\[
\norm{u}{\sim,1/2,\cO}^2
\le
2\omega_{n} \norm{u}{\sim,1/2,\cO'}^2
+
\int_{\cO'} \frac{|u(x)|^2}{\dist(x,\partial\cO')} \dx
\le
(2\omega_{n}+1) \norm{u}{\sim,1/2,\cO'}^2.
\]
This proves part~\ref{ite:A nor Ome Omep}.

Part~\ref{ite:A rev ine} is proved by the following counter-example, which is a modification of the
counter-example in the appendix of~\cite{AinGuo00}. Consider~$D$ to be the upper half of
the unit disk in the $(x,y)$-plane, i.e.,
\[
D = \set{(r,\theta)}{r\in[0,1], \ \theta\in[0,\pi]}
\]
where~$(r,\theta)$ denote polar coordinates. Then define
\begin{align*}
\cO &= \set{(r,\theta)}{r\in[0,1], \ \theta=0 \ \text{or} \ \theta=\pi}
       = [-1,1]\times\{0\} 
\\
\cO_{\varepsilon} &= \set{(r,\theta)}{r\in[\varepsilon,3/4], \ \theta=0}
        = [\varepsilon,3/4]\times\{0\}.
\end{align*}

For~$\varepsilon\in(0,1/2)$, define~$U_\varepsilon : D\to\R$ and~$U : D\to\R$ by
\[
U_\varepsilon(r,\theta) =
\begin{cases}
0, \quad & 0 \le r < \varepsilon,
\\
(-\log r)^{-1/2} - (-\log \varepsilon)^{-1/2}, \quad & \varepsilon \le r < 1/2,
\\
(3-4r)
\left(
(\log 2)^{-1/2} - (-\log \varepsilon)^{-1/2}
\right), \quad & 1/2 \le r < 3/4,
\\
0, \quad & 3/4 \le r \le 1,
\end{cases}
\]
and
\[
U(r,\theta) =
\begin{cases}
0, \quad & r=0,
\\
(-\log r)^{-1/2} , \quad & 0 < r < 1/2,
\\
(3-4r) (\log 2)^{-1/2}, \quad & 1/2 \le r < 3/4,
\\
0, \quad & 3/4 \le r \le 1,
\end{cases}
\]
These two functions are first studied in~\cite{AinGuo00}. We now define~$V_\varepsilon :
D\to\R$ by
\[
V_\varepsilon(r,\theta) =
\begin{cases}
U_{\varepsilon}(r,\theta)\cos\theta, \quad & 0\le\theta<\pi/2,
\\
0, \quad & \pi/2 \le\theta\le\pi.
\end{cases}
\]
Let~$u_\varepsilon$ be the trace of~$V_\varepsilon$ on the boundary of~$D$. Then
\[
\supp(u_\varepsilon) = [\varepsilon,3/4] \times \{0\} = \cO_{\varepsilon}.
\]
For~$(r,\theta)\in D$,
\begin{equation}\label{equ:A Ve Ue U}
\begin{alignedat}{2}
\snorm{V_\varepsilon(r,\theta)}{}
&\le
\snorm{U_\varepsilon(r,\theta)}{}
&&\le
\snorm{U(r,\theta)}{}
\\
\Bsnorm{\frac{\partial V_\varepsilon}{\partial r}(r,\theta)}{}
&\le
\Bsnorm{\frac{\partial U_\varepsilon}{\partial r}(r,\theta)}{}
&&\le
\Bsnorm{\frac{\partial U}{\partial r}(r,\theta)}{}
\\
\Bsnorm{\frac{\partial V_\varepsilon}{\partial \theta}(r,\theta)}{}
&\le
\snorm{U_\varepsilon(r,\theta)}{}
&&\le
\snorm{U(r,\theta)}{}
\end{alignedat}
\end{equation}
We note that
\[
\frac{\partial U}{\partial r}(r,\theta) =
\begin{cases}
\frac12 r^{-1} (-\log r)^{-3/2}, \quad & 0 < r < 1/2,
\\
-4 (\log2)^{-1/2}, \quad & 1/2 < r < 3/4,
\\
0, \quad & 3/4 < r < 1,
\end{cases}
\]
so that~$U\in H^1(D)$. Indeed,
\begin{align*}
\iint_D |U(r,\theta)|^2 \dx \dy
&=
\pi \int_0^{1/2} \frac{r}{-\log r} \dr
+
\pi \int_{1/2}^{3/4} \frac{1}{\log2}(3-4r)^2 r\dr < \infty
\end{align*}
and
\begin{align*}
\iint_D \Bsnorm{\frac{\partial U}{\partial r}(r,\theta)}{}^2 \dx\dy
&=
\frac{\pi}{4} \int_0^{1/2} \frac{\dr}{r(-\log r)^{3}}
+
\pi \int_{1/2}^{3/4} \frac{16}{\log2} r\dr < \infty.
\end{align*}
It follows from~\eqref{equ:A Ve Ue U} that~$V_\varepsilon, U_\varepsilon\in H^1(D)$ and
\[
\norm{V_\varepsilon}{1,D} \lesssim
\norm{U_\varepsilon}{1,D} \le
\norm{U}{1,D}, \quad 0<\varepsilon<1/2.
\]
Consequently, by the definition of the Slobodetski norm and the trace theorem
\[
\norm{u_\varepsilon}{1/2,\cO} \le
\norm{u_\varepsilon}{1/2,\partial D} \lesssim
\norm{V_\varepsilon}{1,D} \lesssim 1.
\]
Since~$\supp(u_\varepsilon)=[\varepsilon,3/4]\times\{0\}$, we have
\begin{align*}
\norm{u_\varepsilon}{\sim,1/2,\cO}^2
&=
\snorm{u_\varepsilon}{1/2,\cO}^2
+
\int_{-1}^1 \frac{|u_\varepsilon(r,0)|^2}{\min\{1-r,1+r\}} \dr
\\
&=
\snorm{u_\varepsilon}{1/2,\cO}^2
+
\int_{\varepsilon}^{3/4} \frac{|u_\varepsilon(r,0)|^2}{\min\{1-r,1+r\}} \dr.
\end{align*}
Due to
\[
\norm{u_\varepsilon}{L^{2}(\cO)}^2
=
\norm{u_\varepsilon}{L^{2}(\cO_\varepsilon)}^2
\simeq
\int_{\varepsilon}^{3/4} \frac{|u_\varepsilon(r,0)|^2}{\min\{1-r,1+r\}} \dr,
\]
we deduce
\[
\norm{u_\varepsilon}{\sim,1/2,\cO}^2
\simeq
\snorm{u_\varepsilon}{1/2,\cO}^2
+
\norm{u_\varepsilon}{L^{2}(\cO)}^2
=
\norm{u_\varepsilon}{1/2,\cO}^2,
\]
so that~$\norm{u_\varepsilon}{\sim,1/2,\cO}^2 \lesssim 1$.
On the other hand, a simple calculation reveals that
\begin{align*}
\norm{u_\varepsilon}{\sim,1/2,\cO_\varepsilon}^2
&\ge
\int_{\cO_\varepsilon} \frac{|u_\varepsilon(r,0)|^2}{\dist(r,\partial\cO_\varepsilon)} \dr
>
\int_{\varepsilon}^{1/2} \frac{|u_\varepsilon(r,0)|^2}{r} \dr
\\
&=
\log|\log\varepsilon| + \frac{4(\log2)^{1/2}}{(\log(1/\varepsilon))^{1/2}}
- \frac{\log2}{\log(1/\varepsilon)} - \log|\log 2| - 3.
\end{align*}
Hence, $\norm{u_\varepsilon}{\sim,1/2,\cO_\varepsilon}\to\infty$ as~$\varepsilon\to0^+$, 
while~$\norm{u_\varepsilon}{\sim,1/2,\cO}$ is bounded. This proves part~\ref{ite:A rev ine},
completing the proof of the theorem.
\end{proof}

We now prove the claim~\eqref{equ:A Ome Omep dis}.
\begin{lemma}\label{lem:A Ome Omep}
Let $\cO$ and~$\cO'$ be two open bounded domains in~$\R^n$, $n\ge1$,
satisfying~$\cO'\subsetneq\cO$, and let~$x\in\cO'$. 
\begin{enumerate}
\renewcommand{\labelenumi}{\theenumi}
\renewcommand{\theenumi}{{\rm (\roman{enumi})}}
\item\label{ite:A xmy dis}
The following inequality holds
\[
\int_{\cO\setminus\cO'} \frac{\dy}{|x-y|^{n+1}}
\le
\frac{\omega_n}{\dist(x,\pa\cO')}
\]
where~$\omega_n$ is the surface area of the unit ball in~$\R^n$.
\item\label{ite:A xmy dis opp}
The opposite inequality is not true, i.e., there is no constant~$C$ independent of~$x$
and~$\cO'$ such that
\[
\frac{1}{\dist(x,\pa\cO')}
\le C
\int_{\cO\setminus\cO'} \frac{\dy}{|x-y|^{n+1}}.
\]
\end{enumerate}
\end{lemma}
\begin{proof}
To prove part~\ref{ite:A xmy dis}, we first observe that
\[
\big(\cO\setminus\cO'\big) \subset \big(\R^n \setminus B_{\rho(x)}(x)\big),
\quad x\in\cO',
\]
where~$\rho(x)=\dist(x,\partial\cO')$ and~$B_{\rho(x)}(x)$ is the ball centred
at~$x$ having radius~$\rho(x)$. By using spherical coordinates centred at~$x$ we
obtain
\begin{align*}
\int_{\cO\setminus\cO'} \frac{\dy}{|x-y|^{n+1}}
&\le
\int_{\R^n \setminus B_{\rho(x)}(x)} \frac{\dy}{|x-y|^{n+1}}
=
\omega_n \int_{\rho(x)}^\infty \frac{r^{n-1}\dr}{r^{n+1}}
=
\frac{\omega_n}{\dist(x,\pa\cO')}.
\end{align*}

To prove part~\ref{ite:A xmy dis opp}, we revisit the proof of part~\ref{ite:A nor Ome Omep} of
Theorem~\ref{the:A Sob loc nor}. If the opposite of~\eqref{equ:A Ome Omep dis} holds, then
\[
\int_{\cO\setminus\cO'} \frac{\dy}{|x-y|^{n+1}}
\simeq
\frac{1}{\dist(x,\pa\cO')}
\]
with constants independent of~$x$ and~$\cO'$. Consider a function~$u$ satisfying the
assumptions in Theorem~\ref{the:A Sob loc nor}. It follows from the definition 
of~$\norm{u}{\sim,1/2,\cO'}$ that
\begin{align*}
\norm{u}{\sim,1/2,\cO'}^2
&\simeq
\snorm{u}{1/2,\cO'}^2 +
\int_{\cO'} \Big( \int_{\cO\setminus\cO'} \frac{\dy}{|x-y|^{n+1}} \Big) |u(x)|^2 \dx.
\end{align*}
This together with equation~\eqref{equ:u12 Omg Omp} yields
\[
\norm{u}{\sim,1/2,\cO'}^2
\simeq
\snorm{u}{1/2,\cO}^2
\le
\norm{u}{\sim,1/2,\cO}^2,
\]
which contradicts part~\ref{ite:A rev ine} of Theorem~\ref{the:A Sob loc nor}.
\end{proof}

\subsection{Proof of claim~\eqref{equ:A Hs Hss}}\label{subsec:pro inc}
To prove the proper inclusion~$\Ht{s}{\cO_1}\subsetneq\Hta{s}{\cO_1}$, it
suffices to prove the following lemma with~$r=1$.
First we recall the definition of the H\"older space~$C^{0,\alpha}(\ol\cO)$ for~$\alpha\in(0,1)$
\[
C^{0,\alpha}(\ol\cO) :=
\bigset{u\in C(\ol\cO)}{\sup_{x,y\in\cO \atop
x\not=y}\frac{|u(x)-u(y)|}{|x-y|^\alpha}<\infty}.
\]

\begin{lemma}\label{lem:A JJ1}
Assume that~$\cO$ has a smooth boundary. Let~$u\in A_0 \cap C^{0,\alpha}(\ol\cO)$
for some~$\alpha\in(0,1)$. Assume that~$u\ge0$, $u\notequiv0$ in~$\cO_1$. Then, for any~$t>0$,
\[
K(t,u) < K_1(t,u)
\]
where~$K$ and~$K_1$ are defined by~\eqref{equ:Ktu} and~\eqref{equ:K1tu}, respectively.
\end{lemma}
\begin{proof}
Fix~$t>0$. There exists~$(u_0^\ast,u_1^\ast)\in\cX(u)$, see~\eqref{equ:Xu}, satisfying
\begin{equation}\label{equ:A opt pro}
(u_0^\ast,u_1^\ast) = \argmin\set{J(t,u_0,u_1)}{(u_0,u_1)\in\cX(u)}
\end{equation}
where~$J(t,u_0,u_1)$ is defined in~\eqref{equ:A JJ1}. Indeed, let
\[
m := \inf\set{J(t,u_0,u_1)}{(u_0,u_1)\in\cX(u)}
= \inf\set{J(t,u-u_1,u_1)}{u_1\in A_1}
\]
and let~$\{u_1^{(k)}\}$ be a sequence in~$A_1$ satisfying
\[
\lim_{k\to\infty} J(t,u-u_1^{(k)},u_1^{(k)}) = m.
\]
It follows from the definition of the functional~$J$ that~$\{u_1^{(k)}\}$ is bounded
in~$H_0^1(\cO)$. Hence there exists a subsequence~$\{u_1^{(k_j)}\}$ converging
weakly in~$H_0^1(\cO)$ to~$u_1^\ast$. It can be seen that~$u_1^\ast\in A_1$. From the
definition of~$m$, it follows that
\[
m\le J(t,u-u_1^\ast,u_1^\ast). 
\]
On the other hand, recall the property that if a sequence~$\{v_j\}$ in a Banach
space~$E$ converges weakly to~$v\in E$, 
then~$\norm{v}{E} \le \liminf_{j\to\infty}\norm{v_j}{E}$. Hence,
the weak convergence of~$\{u_1^{(k_j)}\}$ to~$u_1^\ast$ implies
\[
J(t,u-u_1^\ast,u_1^\ast)
\le
\liminf_{j\to\infty}
J(t,u-u_1^{(k_j)},u_1^{(k_j)}) = m.
\]
Thus~$J(t,u-u_1^\ast,u_1^\ast)=m$ and therefore~$(u_0^\ast,u_1^\ast):=(u-u_1^\ast,u_1^\ast)$ 
is a minimiser.

To prove the lemma, it suffices to show that~$(u_0^\ast,u_1^\ast)\notin\cX_1(u)$;
see~\eqref{equ:Xu1}. The problem~\eqref{equ:A opt pro} is an optimisation problem with
constraint, the constraint being~$g(u_0,u_1)=0$ where
\[
g : L^2(\cO) \times H_0^1(\cO) \to L^2(\cO), \quad 
g(u_0,u_1) := u_0 + u_1 - u.
\]
For each~$t>0$, the Lagrangian functional $\cL:L^2(\cO)\times H_0^1(\cO)\times L^2(\cO)
\to\R$ is defined by
\[
\cL(u_0,u_1;p) := J(t,u_0,u_1) + G(u_0,u_1,p)
\]
where
\(
G(u_0,u_1,p) := \inpro{p}{g(u_0,u_1)}_{L^2(\cO)}
= \inpro{p}{u_0+u_1-u}_{L^2(\cO)}.
\)
Here~$p$ is the Lagrangian multiplier. 
It is well known that
\begin{equation}\label{equ:min vvp}
\min_{(u_0,u_1)\in\cX(u)} J(t,u_0,u_1)
=
\inf_{u_0\in L^2(\cO), \; u_1\in H_0^1(\cO)} \
\sup_{p\in L^2(\cO)} \cL(u_0,u_1,p).
\end{equation}

For any functional~$F:(u,p)\mapsto F(u,p)$, we denote by 
\(
\pa_u F(u,p)(\varphi)
\)
the $u$-Fr\'echet derivative of~$F$ at~$(u,p)$, acting on~$\varphi$. Similarly,~$\pa_p
F(u,p)(q)$ denotes the $p$-Fr\'echet derivative of~$F$ at~$(u,p)$, acting on~$q$. The
minimiser~$(u_0^\ast,u_1^\ast)$ to problem~\eqref{equ:A opt pro}
and the solution~$(u_0^\ast,u_1^\ast,p^\ast)$ to the minimax problem~\eqref{equ:min vvp} solve
the following equations
\begin{align*}
\partial_{u_0} \cL(u_0,u_1,p) = 0, 
\quad
\partial_{u_1} \cL(u_0,u_1,p) = 0,
\quad
\partial_p \cL(u_0,u_1,p) = 0.
\end{align*}
Since (see e.g. \cite{Car67})
\begin{alignat*}{2}
\partial_{u_0}J(t,u_0,u_1)(\varphi)
&=
2\inpro{u_0}{\varphi}_{L^2(\cO)} &&\quad\forall\varphi\in L^2(\cO),
\\
\partial_{u_1}J(t,u_0,u_1)(\psi)
&=
2t^2\inpro{\nabla u_1}{\nabla \psi}_{L^2(\cO)} &&\quad\forall\psi\in H_0^1(\cO),
\\
\partial_{u_0} G(u_0,u_1,p)(\varphi)
&=
\inpro{p}{\varphi}_{L^2(\cO)} &&\quad\forall \varphi \in L^2(\cO),
\\
\partial_{u_1} G(u_0,u_1,p)(\varphi)
&=
\inpro{p}{\psi}_{L^2(\cO)} &&\quad\forall \psi \in H_0^1(\cO),
\\
\partial_p G(u_0,u_1,p)(q)
&=
\inpro{q}{u_0+u_1-u}_{L^2(\cO)} &&\quad\forall q\in L^2(\cO),
\end{alignat*}
we have
\begin{alignat*}{2}
\partial_{u_0} \cL(u_0,u_1,p)(\varphi)
&=
2\inpro{u_0}{\varphi}_{L^2(\cO)}
+
\inpro{p}{\varphi}_{L^2(\cO)}
&&\quad\forall\varphi\in L^2(\cO),
\\
\partial_{u_1} \cL(u_0,u_1,p)(\psi)
&=
2t^2\inpro{\nabla u_1}{\nabla\psi}_{L^2(\cO)}
+
\inpro{p}{\psi}_{L^2(\cO)}
&&\quad\forall\psi\in H_0^1(\cO),
\\
\partial_{p} \cL(u_0,u_1,p)(q)
&=
\inpro{q}{u_0+u_1-u}_{L^2(\cO)}
&&\quad\forall q\in L^2(\cO).
\end{alignat*}
Hence, $(u_0^\ast,u_1^\ast,p^\ast)\in L^2(\cO)\times H_0^1(\cO)\times L^2(\cO)$ satisfies
\begin{align}
2\inpro{u_0^\ast}{\varphi}_{L^2(\cO)} + \inpro{p^\ast}{\varphi}_{L^2(\cO)} &= 0
\quad\forall\varphi\in L^2(\cO),
\label{equ:A Lv0}
\\
2t^2\inpro{\nabla u_1^\ast}{\nabla\psi}_{L^2(\cO)} + \inpro{p^\ast}{\psi}_{L^2(\cO)} &= 0
\quad\forall\psi\in H_0^1(\cO),
\label{equ:A Lv1}
\\
\inpro{q}{u_0^\ast+u_1^\ast-u}_{L^2(\cO)} &= 0
\quad \forall q\in L^2(\cO).
\label{equ:A Lpq}
\end{align}
It follows from~\eqref{equ:A Lv0} and~\eqref{equ:A Lv1} that
\[
t^2\inpro{\nabla u_1^\ast}{\nabla\psi}_{L^2(\cO)} -
\inpro{u_0^\ast}{\psi}_{L^2(\cO)} = 0
\quad\forall\psi\in H_0^1(\cO).
\]
This equation and~\eqref{equ:A Lpq} give
\begin{equation*}\label{equ:A Jv wea for}
t^2\inpro{\nabla u_1^\ast}{\nabla\psi}_{L^2(\cO)} +
\inpro{u_1^\ast}{\psi}_{L^2(\cO)} = 
\inpro{u}{\psi}_{L^2(\cO)} \quad\forall\psi\in H_0^1(\cO).
\end{equation*}
This is a weak formulation of the following boundary value problem
\begin{equation}\label{equ:A Poi bvp}
\begin{aligned}
-t^2 \Delta u_1^\ast + u_1^\ast &= u \quad\text{in}\quad\cO,
\\
u_1^\ast &= 0 \quad\text{on}\quad\partial\cO.
\end{aligned}
\end{equation}
Since~$\cO$ has smooth boundary and~$u\in C^{0,\alpha}(\ol\cO)$, we deduce
that~$u_1^\ast\in C(\ol\cO)\cap C^2(\cO)$. Moreover, since~$u\ge0$ in~$\cO$,
due to the strong maximum principle, see e.g. \cite[Corollary~9.37]{Bre11},
either~$u_1^\ast>0$ in~$\cO$ or~$u_1^\ast\equiv0$ in~$\cO$. If~$u_1^\ast\equiv0$
then~$u_0^\ast=u$. Moreover, \eqref{equ:A Lv1} implies~$p\equiv0$ so
that~$u_0^\ast\equiv0$ on~$\cO$ due to~\eqref{equ:A Lv0}. This contradicts the
assumption that~$u\notequiv0$. Hence~$u_1^\ast>0$ on~$\cO$, which
implies~$(u_0^\ast,u_1^\ast)\notin\cX_1(u)$.
\end{proof}

\begin{remark}
The above result is consistent with the well-known fact that the values of the
solution~$u_1^\ast$ of~\eqref{equ:A Poi bvp} in a subdomain~$\cO_2\subsetneq\cO$ depends
on the values of~$u$ not only in~$\cO_2$ but in all of~$\cO$; see
e.g. \cite[page~307]{Bre11}.
\end{remark}

\subsection{Proof of Theorem~\ref{the:A tvp 2}}

The proof follows along the lines of the proof of \cite[Theorem~4.1]{AinMcLTra}.
\begin{proof}
Introduce the product space
\[
\wtd\Pi^s := \prod_{j=1}^N \wtd H_{\ast}^{s}(\Omega_j),
\quad 0 \le s \le r,
\]
with a norm defined from the interpolation norms by
\[
\norm{\vecu}{\wtd\Pi^s}^2
:=
\sum_{j=1}^N
\norm{u_j}{\Hta{s}{\Omega_j}}^2
=
\sum_{j=1}^N
\norm{u_j}{\Ht{s}{\Omega}}^2,
\]
where~$\vecu=(u_1,\ldots,u_N)$. If~$s=\theta r$ for some $\theta\in(0,1)$, then
\[
\wtd\Pi^s = [\wtd\Pi^0,\wtd\Pi^r]_{\theta}.
\]
On the product set~$\wtd\Pi^s$, consider the sum operator~$\cS:\wtd\Pi^s\to\wtd
H^s(\Omega)$ defined by
\[
\cS\vecu
:=
\sum_{j=1}^N u_j, \quad u_j\in\Hta{s}{\Omega_j}.
\]
Recalling
that~$\norm{\cdot}{\Hta{0}{\Omega_j}}=\norm{\cdot}{L^2(\Omega)}$
and~$\norm{\cdot}{\Hta{r}{\Omega_j}}=\snorm{\cdot}{H^r(\Omega)}$, we deduce
\[
\norm{\cS\vecu}{\Ht{s}{\Omega}}^2
=
\sum_{j=1}^N \norm{u_j}{\Hta{s}{\Omega_j}}^2
=
\norm{\vecu}{\wtd\Pi^s}^2, \quad s=0 \ \text{ or } \ s=r.
\]
By interpolation
\[
\norm{\cS\vecu}{\Ht{s}{\Omega}}
\le
\norm{\vecu}{\wtd\Pi^s} \quad\text{for}\quad 0\le s \le r. 
\]

Now for any function~$u\in\Ht{s}{\Omega}$ such that~$u_j\in\Hta{s}{\Omega_j}$,
$j=1,\ldots,N$, where~$u_j$ is the zero extension of~$u|_{\Omega_j}$ 
onto~$\Omega\setminus\ol\Omega_j$, we define~$\vecu=(u_1,\ldots,u_N)$ 
Then $u = \cS\vecu$ because~$\{\Omega_1,\ldots,\Omega_N\}$ is a partition
of~$\Omega$. Consequently
\[
\norm{u}{\Ht{s}{\Omega}}^2
=
\norm{\cS\vecu}{\Ht{s}{\Omega}}^2
\le
\norm{\vecu}{\wtd\Pi^s}^2
=
\sum_{j=1}^N
\norm{u_j}{\Ht{s}{\Omega}}^2,
\]
proving~\eqref{equ:A tvp3}.
\end{proof}

\section{Applications}\label{sec:app}

Proposition~\ref{pro:A 2 nor 3 nor} and Corollary~\ref{cor:tvp} are needed in the
analysis of domain decomposition methods for boundary integral equations. Consider for
example the exterior Neumann boundary value problem
\begin{equation}\label{equ:bvp}
\begin{alignedat}{2}
-\Delta U &= 0 &&\quad\text{in } \R^3\setminus\ol\Omega,
\\
\frac{\pa U}{\pa n_i} &= g_i &&\quad\text{on } \Omega_i, \ i=1,2,
\\
\frac{\pa U}{\pa r} &= o(1/r) &&\quad\text{as } r=|x|\to\infty,
\end{alignedat}
\end{equation}
where~$\Omega$ is a screen in~$\R^3$ and~$\Omega_i$, $i=1,2$, are two sides of~$\Omega$
determined by two opposite normal vectors~$n_i$. It is well known that~\cite{Ste87,SteWen84}
if~$\varphi:=[U]_\Omega$ denotes the jump of~$U$ across the screen~$\Omega$,
then~\eqref{equ:bvp} is equivalent to the boundary integral equation
\begin{equation}\label{equ:hyp}
\hyp\varphi(x) = -g(x), \quad x\in\Omega,
\end{equation}
where~$\hyp$ is the hypersingular integral operator defined by
\[
\hyp\varphi(x) := -\frac{1}{2\pi} \frac{\pa}{\pa n_x} \int_{\Omega} 
\frac{\pa}{\pa n_y} \Big( \frac{1}{|x-y|} \Big) \varphi(y) \ds_y.
\]
It is also well know that \cite{Cos88,Ste87,SteWen84}
that~$\hyp:\Ht{1/2}{\Omega}\to\Hs{-1/2}{\Omega}$ is bijective, where~$\Hs{-1/2}{\Omega}$
is the dual of~$\Ht{1/2}{\Omega}$ with respect to the $L^2$-dual pairing. A weak
formulation for equation~\eqref{equ:hyp} is finding~$\varphi\in\Ht{1/2}{\Omega}$ satisfying
\begin{equation}\label{equ:wea for}
a(\varphi,\psi) = -\inpro{g}{\psi} \quad\forall\psi\in\Ht{1/2}{\Omega}
\end{equation}
where $\inpro{\cdot}{\cdot}$ denotes the $L^2$-inner product and the bilinear 
form~$a(\cdot,\cdot)$ is defined
by~$a(\phi,\psi)=\inpro{\hyp\phi}{\psi}$ for all~$\phi,\psi\in\Ht{1/2}{\Omega}$. 
It is known that this bilinear form defines a norm which is equivalent to the
$\Ht{1/2}{\Omega}$-norm, i.e.,
\begin{equation}\label{equ:bil for}
a(\psi,\psi) \simeq \norm{\psi}{\Ht{1/2}{\Omega}}^2
\quad\forall\psi\in\Ht{1/2}{\Omega}.
\end{equation}
Together with~\eqref{equ:H12 Sob} this implies
\begin{equation}\label{equ:Slo nor}
a(\psi,\psi) \simeq \norm{\psi}{\sim,1/2,\Omega}^2
\quad\forall\psi\in\Ht{1/2}{\Omega}.
\end{equation}

The Galerkin boundary element method applied to equation~\eqref{equ:wea for} results in the
following equation which computes an approximate solution~$\varphi_M\in\cV_M$
\begin{equation}\label{equ:Gal equ}
a(\varphi_M,\psi_M) = -\inpro{g}{\psi_M} \quad\forall\psi_M\in\cV_M
\end{equation}
where~$\cV_M$ is an $M$-dimensional subspace of~$\Ht{1/2}{\Omega}$. It is known that
when~$M\to\infty$ the approximate solution~$\varphi_M$ converges to~$\varphi$
in~$\Ht{1/2}{\Omega}$; see e.g.~\cite{GwiSte18}.

Let~$\{\phi_1,\ldots,\phi_M\}$ be a basis for~$\cV_M$. Then~\eqref{equ:Gal equ} is equivalent to
\[
a(\varphi_M,\phi_j) = -\inpro{g}{\phi_j}, \quad j=1,\ldots,M.
\]
By representing~$\varphi_M$ as~$\varphi_M=\sum_{i=1}^M c_i \phi_i$ we deduce
from the above equations a system of linear equations written in matrix form as
\begin{equation}\label{equ:Axb}
\vecA\vecx = \vecb.
\end{equation}
Here~$\vecA$ is a symmetric matrix of size~$M\times M$ with
entries~$A_{ij}=a(\phi_i,\phi_j)$, $i,j=1,\ldots,M$; the unknown vector~$\vecx$ has 
entries~$x_i=c_i$, and the right-hand side vector~$\vecb$ has
entries~$b_i=-\inpro{g}{\phi_i}$, $i=1,\ldots,M$. 

The matrix~$\vecA$ is positive definite, see e.g.~\cite{GwiSte18}, which guarantees the unique
solution of~\eqref{equ:Axb}. However, when~$M$ is large (which is common in real-life
applications), the quality of the solution of this
equation given by a computer depends on the property of~$\vecA$. We briefly explain this
phenomenon and refer the reader to any textbook on numerical linear algebra for detail.

Let~$\lambda_{\max}(\vecA)$ and~$\lambda_{\min}(\vecA)$ be the maximum and minimum eigenvalues
of~$\vecA$, respectively, and let the condition number~$\kappa(\vecA)$ be defined
by~$\kappa(\vecA)=\lambda_{\max}(\vecA)/\lambda_{\min}(\vecA)$.
The matrix~$\vecA$ is ill-conditioned, namely~$\kappa(\vecA)$ increases significantly with
the size of~$\vecA$. More precisely, it is known that~$\kappa(\vecA)=O(M)$; see e.g.
\cite{AinMcLTra}. Assume that a direct solver (by Gaussian elimination or row reduction
algorithm) is used to solve~\eqref{equ:Axb} on a computer. Due to round-off errors (a computer
can only work with a limited number of digits after the decimal point), the computer
only yields an approximate solution~$\vecx^\ast$ to~$\vecx$. It is known that the
relative error of this approximation is proportional to~$\kappa(\vecA)$; see
\cite[Chapter~3]{GolVan13}. Therefore, when~$M$ is large the computed solution~$\vecx^\ast$ is
not reliable. For example, if~$M$ is in the range from~$10^{5}$ to~$10^{8}$ (which is quite common
in real-life applications), one can expect at most from~$7$ to~$10$ correct digits
of~$\vecx^\ast$ after the decimal point (which is not satisfactory in many
applications).

A better approach for solving~\eqref{equ:Axb} with large~$M$ is using iterative methods, for
example, the conjugate gradient method; see \cite[Chapter~11]{GolVan13}. Starting from an initial
guess~$\vecx_0$, this algorithm yields a sequence~$\{\vecx_k\}$ satisfying
\[
\norm{\vecx_k-\vecx}{\vecA}
\le
2 \Big(
\frac{\kappa(\vecA)-1}{\kappa(\vecA)+1}
\Big)^k
\norm{\vecx_0-\vecx}{\vecA}, \quad k=1,2,\ldots.
\]
Here the norm $\norm{\cdot}{\vecA}$ is defined
by~$\norm{\vecx}{\vecA}=\sqrt{\vecx^\top\vecA\vecx}$ with~$\vecx^\top$ being the transpose
of~$\vecx$. Clearly, when~$M$ is large, so is~$\kappa(\vecA)=O(M)$, and thus we expect a very
large number of iterations~$k$ to obtain a satisfactory solution~$\vecx_k$.

To solve the ill-conditioned system~\eqref{equ:Axb} efficiently, a preconditioner~$\vecC$ is
required. Instead of solving~\eqref{equ:Axb}, one solves
\begin{equation*}\label{equ:CAx}
\vecC^{-1}\vecA\vecx=\vecC^{-1}\vecb,
\end{equation*}
with the preconditioner~$\vecC$ designed such that~$\vecC^{-1}\approx\vecA^{-1}$ so
that~$\kappa(\vecC^{-1}\vecA)$ is much smaller than~$\kappa(\vecA)$. Ideally,
$\kappa(\vecC^{-1}\vecA)$ is bounded with respect to~$M$, or it grows at most logarithmically
with respect to~$M$.

Preconditioners by domain decomposition have been studied for~\eqref{equ:Axb}; see
e.g.~\cite{AinGuo00,AinGuo02,FeiFuhPraSte17b,GraMcL06,Heu96b,Heu96a,Heu01,HeuLeySte07,
HeuSte98c,Tra00,TraSte_h,TraSte_p,TraSte01,TraSte04}. 
The method can be briefly described as follows. Partition the domain~$\Omega$ into
subdomains~$\Omega_1$, \ldots, $\Omega_N$. On each subdomain we
define~$\cV_{M,j}=\cV_M\cap\Ht{1/2}{\Omega_j}$, $j=1,\ldots,N$, and decompose~$\cV_M$ by
\begin{equation}\label{equ:ssd}
\cV_M = \cV_{M,1} + \cdots + \cV_{M,N}.
\end{equation}
A preconditioner~$\vecC$ is defined using this subspace decomposition. To estimate the
condition number~$\kappa(\vecC^{-1}\vecA)$, one needs to show, among other things,
the following two statements, see e.g.~\cite[Chapter~2]{SteTra21},
\begin{enumerate}
\renewcommand{\labelenumi}{\theenumi}
\renewcommand{\theenumi}{{\rm (\roman{enumi})}}
\item\label{ite:sta}
For any~$u\in\cV_M$, there exists a decomposition~$u=u_1+\cdots+u_N$ with~$u_j\in\cV_{M,j}$,
$j=1,\ldots,N$, such that
\begin{equation}\label{equ:min}
C_1 \sum_{j=1}^N a(u_j,u_j) \le a(u,u).
\end{equation}
\item\label{ite:coe}
For any~$u\in\cV_M$ and any decomposition~$u=u_1+\cdots+u_N$ with~$u_j\in\cV_{M,j}$, $j=1,\ldots,N$,
the following inequality holds
\begin{equation}\label{equ:max}
a(u,u) \le C_2 \sum_{j=1}^N a(u_j,u_j).
\end{equation}
\end{enumerate}
The positive constants~$C_1$ and~$C_2$ are independent of~$u\in\cV_M$. Ideally, they are also
independent of the parameter~$M$ determining the size of equation~\eqref{equ:Gal equ}. A less
optimal case is when these constants depend
at most logarithmically on~$M$. Statement~\ref{ite:sta} 
yields~$C_1 \le \lambda_{\min}(\vecC^{-1}\vecA)$ and is called the
stability of the decomposition~\eqref{equ:ssd},
while Statement~\ref{ite:coe} yields~$\lambda_{\max}(\vecC^{-1}\vecA)\le C_2$ and 
is called the coercivity of the decomposition. The condition
number~$\kappa(\vecC^{-1}\vecA)$ is then bounded by~$C_2/C_1$.

Due to~\eqref{equ:bil for} and~\eqref{equ:Slo nor}, either~$\norm{\cdot}{\Ht{1/2}{\Omega}}$ 
or~$\norm{\cdot}{\sim,1/2,\Omega}$ can be used to prove~\eqref{equ:min}
and~\eqref{equ:max}. It turns out that the norm~$\norm{\cdot}{\sim,1/2,\Omega}$ is more
suitable to prove~\eqref{equ:min} while~$\norm{\cdot}{\Ht{1/2}{\Omega}}$ is good
for proving~\eqref{equ:max}. There has been a misconception that
\begin{equation}\label{equ:wro equ}
a(u_j,u_j)\simeq\norm{u_j}{\Ht{1/2}{\Omega_j}}^2\simeq\norm{u_j}{\sim,1/2,\Omega_j}^2
\end{equation}
with universal constants independent of the diameter of~$\Omega_j$. Thus, ubiquitously in the
literature, a common practice has been to use Theorem~\ref{the:A tvp} to prove
\[
\norm{u}{\Ht{1/2}{\Omega}}^2 \le C_2 \sum_{j=1}^N \norm{u_j}{\Ht{1/2}{\Omega_j}}^2
\]
in order to derive~\eqref{equ:max}. Theorem~\ref{the:A Sob loc nor} and 
Proposition~\ref{pro:A 2 nor 3 nor} imply that the equivalences~\eqref{equ:wro equ} hold with
constants depending on the diameter of~$\Omega_j$, which is proportional to~$M$. This may result
in polynomial dependence on~$M$ of the constant~$C_2$. To avoid this unsatisfactory
result, one has to prove
\[
\norm{u}{\Ht{1/2}{\Omega}}^2 \le C_2 \sum_{j=1}^N \norm{u_j}{\Ht{1/2}{\Omega}}^2.
\]
This inequality can be obtained by invoking Corollary~\ref{cor:tvp}.

\begin{remark}
After this article was submitted, the author derived similar results for the Sobolev
space~$\Ht{s}{\Omega}$ of negative order~$s$. These results were reported in
Appendix~A of the monograph~\cite{SteTra21}.
\end{remark}
\section*{Acklowledgement}
We thank the anonymous referees for their constructive remarks. In particular a remark
shortens the proof of Lemma~\ref{lem:A Ome Omep} part~\ref{ite:A xmy dis}.
\bibliographystyle{myabbrv}
\bibliography{mybib}

\newcommand{\noopsort}[1]{}\def\cprime{$'$}
  \def\soft#1{\leavevmode\setbox0=\hbox{h}\dimen7=\ht0\advance \dimen7
  by-1ex\relax\if t#1\relax\rlap{\raise.6\dimen7
  \hbox{\kern.3ex\char'47}}#1\relax\else\if T#1\relax
  \rlap{\raise.5\dimen7\hbox{\kern1.3ex\char'47}}#1\relax \else\if
  d#1\relax\rlap{\raise.5\dimen7\hbox{\kern.9ex \char'47}}#1\relax\else\if
  D#1\relax\rlap{\raise.5\dimen7 \hbox{\kern1.4ex\char'47}}#1\relax\else\if
  l#1\relax \rlap{\raise.5\dimen7\hbox{\kern.4ex\char'47}}#1\relax \else\if
  L#1\relax\rlap{\raise.5\dimen7\hbox{\kern.7ex
  \char'47}}#1\relax\else\message{accent \string\soft \space #1 not
  defined!}#1\relax\fi\fi\fi\fi\fi\fi}
\begin{thebibliography}{10}

\bibitem{Ada75}
R.~A. Adams.
\newblock {\em Sobolev Spaces}.
\newblock Academic Press [A subsidiary of Harcourt Brace Jovanovich,
  Publishers], New York-London, 1975.
\newblock Pure and Applied Mathematics, Vol. 65.

\bibitem{AinGuo00}
M.~Ainsworth and B.~Guo.
\newblock An additive {S}chwarz preconditioner for $p$-version boundary element
  approximation of the hypersingular operator in three dimensions.
\newblock {\em Numer. Math.},  {\bf 85} (2000), 343--366.

\bibitem{AinGuo02}
M.~Ainsworth and B.~Guo.
\newblock Analysis of iterative sub-structuring techniques for boundary element
  approximation of the hypersingular operator in three dimensions.
\newblock {\em Appl. Anal.},  {\bf 81} (2002), 241--280.

\bibitem{AinMcLTra}
M.~Ainsworth, W.~McLean, and T.~Tran.
\newblock The conditioning of boundary element equations on locally refined
  meshes and preconditioning by diagonal scaling.
\newblock {\em SIAM J. Numer. Anal.},  {\bf 36} (1999), 1901--1932.

\bibitem{BerLof}
J.~Bergh and J.~L\"ofstr\"om.
\newblock {\em Interpolation Spaces: An Introduction}.
\newblock Springer-Verlag, Berlin, 1976.

\bibitem{Bre11}
H.~Brezis.
\newblock {\em Functional Analysis, {S}obolev Spaces and Partial Differential
  Equations}.
\newblock Universitext. Springer, New York, 2011.

\bibitem{Car67}
H.~Cartan.
\newblock {\em Calcul Diff\'{e}rentiel}.
\newblock Hermann, Paris, 1967.

\bibitem{Cos88}
M.~Costabel.
\newblock Boundary integral operators on {L}ipschitz domains: Elementary
  results.
\newblock {\em SIAM J. Math. Anal.},  {\bf 19} (1988), 613--626.

\bibitem{DupSco80}
T.~Dupont and R.~Scott.
\newblock Polynomial approximation of functions in {S}obolev spaces.
\newblock {\em Math. Comp.},  {\bf 34} (1980), 441--463.

\bibitem{FeiFuhPraSte17b}
M.~Feischl, T.~F\"{u}hrer, D.~Praetorius, and E.~P. Stephan.
\newblock Optimal additive {S}chwarz preconditioning for hypersingular integral
  equations on locally refined triangulations.
\newblock {\em Calcolo},  {\bf 54} (2017), 367--399.

\bibitem{GolVan13}
G.~H. Golub and C.~F. Van~Loan.
\newblock {\em Matrix Computations}.
\newblock Johns Hopkins Studies in the Mathematical Sciences. Johns Hopkins
  University Press, Baltimore, MD, fourth edition, 2013.

\bibitem{GraMcL06}
I.~G. Graham and W.~McLean.
\newblock Anisotropic mesh refinement: the conditioning of {G}alerkin boundary
  element matrices and simple preconditioners.
\newblock {\em SIAM J. Numer. Anal.},  {\bf 44} (2006), 1487--1513.

\bibitem{Gri85}
P.~Grisvard.
\newblock {\em Elliptic Problems in Nonsmooth Domains}.
\newblock Pitman, Boston, 1985.

\bibitem{GwiSte18}
J.~Gwinner and E.~P. Stephan.
\newblock {\em Advanced Boundary Element Methods}, volume~52 of {\em Springer
  Series in Computational Mathematics}.
\newblock Springer, Cham, 2018.
\newblock Treatment of boundary value, transmission and contact problems.

\bibitem{Heu96b}
N.~Heuer.
\newblock Additive {S}chwarz methods for weakly singular integral equations in
  {$\mathbb{R}^3$} -- the $p$ version.
\newblock In {\em Boundary Elements: Implementation and Analysis of Advanced
  Algorithms (Proc. of the 12th GAMM--Seminar, Kiel, Germany, January, 1996)},
  W.~Hackbusch and G.~Wittum, editors, pages 126--135, Braunschweig, 1996.
  Vieweg-Verlag.

\bibitem{Heu96a}
N.~Heuer.
\newblock Efficient algorithms for the $p$ version of the boundary element
  method.
\newblock {\em J. Integral Equations Appl.},  {\bf 8} (1996), 337--361.

\bibitem{Heu01}
N.~Heuer.
\newblock Additive {S}chwarz method for the {$p$}-version of the boundary
  element method for the single layer potential operator on a plane screen.
\newblock {\em Numer. Math.},  {\bf 88} (2001), 485--511.

\bibitem{Heu14}
N.~Heuer.
\newblock On the equivalence of fractional-order {S}obolev semi-norms.
\newblock {\em J. Math. Anal. Appl.},  {\bf 417} (2014), 505--518.

\bibitem{HeuLeySte07}
N.~Heuer, F.~Leydecker, and E.~P. Stephan.
\newblock An iterative substructuring method for the {$hp$}-version of the
  {BEM} on quasi-uniform triangular meshes.
\newblock {\em Numer. Methods Partial Differential Equations},  {\bf 23}
  (2007), 879--903.

\bibitem{HeuSte98c}
N.~Heuer and E.~P. Stephan.
\newblock Iterative substructuring for hypersingular integral equations in
  {$\bold R^3$}.
\newblock {\em SIAM J. Sci. Comput.},  {\bf 20} (1998), 739--749.

\bibitem{LioMag72}
J.~L. Lions and E.~Magenes.
\newblock {\em Non-Homogeneous Boundary Value Problems and Applications I}.
\newblock Springer-Verlag, New York, 1972.

\bibitem{Ste87}
E.~P. Stephan.
\newblock Boundary integral equations for screen problems in {$\mathbb{R}^3$}.
\newblock {\em Integral Equations Operator Theory},  {\bf 10} (1987), 236--257.

\bibitem{SteTra21}
E.~P. Stephan and T.~Tran.
\newblock {\em Schwarz Methods and Multilevel Preconditioners for Boundary
  Element Methods}.
\newblock Springer, (to appear).

\bibitem{SteWen84}
E.~P. Stephan and W.~L. Wendland.
\newblock An augmented {G}alerkin procedure for the boundary integral method
  applied to two-dimensional screen and crack problems.
\newblock {\em Appl. Anal.},  {\bf 18} (1984), 183--219.

\bibitem{Tra00}
T.~Tran.
\newblock Overlapping additive {S}chwarz preconditioners for boundary element
  methods.
\newblock {\em J. Integral Eqns Appl.},  {\bf 12} (2000), 177--207.

\bibitem{TraSte_h}
T.~Tran and E.~P. Stephan.
\newblock Additive {S}chwarz methods for the $h$ version boundary element
  method.
\newblock {\em Appl. Anal.},  {\bf 60} (1996), 63--84.

\bibitem{TraSte_p}
T.~Tran and E.~P. Stephan.
\newblock Additive {S}chwarz algorithms for the $p$ version of the {G}alerkin
  boundary element method.
\newblock {\em Numer. Math.},  {\bf 85} (2000), 433--468.

\bibitem{TraSte01}
T.~Tran and E.~P. Stephan.
\newblock Two-level additive {S}chwarz preconditioners for the {$h$}-{$p$}
  version of the {G}alerkin boundary element method for 2-d problems.
\newblock {\em Computing},  {\bf 67} (2001), 57--82.

\bibitem{TraSte04}
T.~Tran and E.~P. Stephan.
\newblock An overlapping additive {S}chwarz preconditioner for boundary element
  approximations to the {L}aplace screen and {L}am\'e crack problems.
\newblock {\em J. Numer. Math.},  {\bf 12} (2004), 311--330.

\bibitem{Pet}
T.~von Petersdorff.
\newblock {\em Randwertprobleme der {E}lastizit{\"a}tstheorie f{\"u}r
  {P}olyeder--{S}ing\-ular\-it{\"a}ten und {A}pproximation mit
  {R}andelementmethoden}.
\newblock PhD thesis, Technische Hochschule Darmstadt, Darmstadt, 1989.

\end{thebibliography}

\end{document}